\documentclass[11pt]{article}
\usepackage{amsmath,amssymb,amsthm}
\usepackage{graphicx,subfigure,float,url,color}
\usepackage{mathrsfs}
\usepackage[colorlinks=true]{hyperref}
\usepackage{pdfsync}

\graphicspath{{fig/}}

\normalbaselines

\newtheorem{Theorem}{Theorem}[section]
\newtheorem{Definition}[Theorem]{Definition}

\newtheorem{Lemma}[Theorem]{Lemma}
\newtheorem{Corollary}[Theorem]{Corollary}

\newtheorem{Remark}[Theorem]{Remark}

\newcommand{\N}{\mathbb N}
\newcommand{\RR}{{{\rm I} \kern -.15em {\rm R} }}
\newcommand{\C}{{{\rm l} \kern -.42em {\rm C} }}
\newcommand{\nat}{{{\rm I} \kern -.15em {\rm N} }}

\newcommand{\be}{\begin{equation}}
\newcommand{\ee}{\end{equation}}
\newcommand{\beq}{\begin{eqnarray}}
\newcommand{\eeq}{\end{eqnarray}}
\newcommand{\beqs}{\begin{eqnarray*}}
\newcommand{\eeqs}{\end{eqnarray*}}
\newcommand{\bt}{\begin{Theorem}}
\newcommand{\et}{\end{Theorem}}
\newcommand{\br}{\begin{Remark}}
\newcommand{\er}{\end{Remark}}
\newcommand{\bc}{\begin{Corollary}}
\newcommand{\ec}{\end{Corollary}}
\newcommand{\bl}{\begin{Lemma}}
\newcommand{\el}{\end{Lemma}}
\newcommand{\bd}{\begin{definition}}
\newcommand{\ed}{\end{definition}}

\renewcommand{\geq}{\geqslant}
\renewcommand{\ge}{\geqslant}
\renewcommand{\leq}{\leqslant}
\renewcommand{\le}{\leqslant}

\title{Convergence to consensus of the general finite-dimensional Cucker-Smale model with time-varying delays}
\author{
Cristina Pignotti\footnote{Dipartimento di Ingegneria e Scienze dell'Informazione e Matematica, Universit\`{a} di L'Aquila, Via Vetoio, Loc. Coppito, 67010 L'Aquila Italy (\texttt{pignotti@univaq.it}).}
\and Emmanuel Tr\'elat \footnote{Sorbonne Universit\'es, UPMC Univ Paris 06, CNRS UMR 7598, Laboratoire Jacques-Louis Lions, 4 place Jussieu, 75005, Paris, France (\texttt{emmanuel.trelat@upmc.fr}).}
}

\date{}

\begin{document}

\maketitle

\begin{abstract}
We consider the celebrated Cucker-Smale model in finite dimension, modelling interacting collective dynamics and their possible evolution to consensus. The objective of this paper is to study the effect of time delays in the general model.
By a Lyapunov functional approach, we provide convergence results to consensus for symmetric as well as nonsymmetric communication weights under some structural conditions.
%The results are obtained by performing both a $L^2$ analysis and a $L^\infty$ analysis for the  symmetric or nonsymmetric case respectively.
\end{abstract}

\vspace{5 mm}

\def\qed{\hbox{\hskip 6pt\vrule width6pt
height7pt
depth1pt  \hskip1pt}\bigskip}

%% {\bf 2000 Mathematics Subject Classification:}
%%35L05, 93D15

 %%{\bf Keywords and Phrases:}  wave equation,  delay feedbacks, stabilization

\section{Introduction}
%\label{pbform}
%\hspace{5mm}

\setcounter{equation}{0}
The study of collective behavior of autonomous agents has recently attracted great interest in various scientific applicative areas, such as biology, sociology, robotics, economics (see \cite{Axe, Aydogdu, Bullo, Camazine, CFTV, Couzin, HK, Hel, Jack, Krause, Parrish, Perea, Toner}). The main motivation is to model and explain the possible emergence of self-organization or global pattern formation in a large group of agents having mutual interactions, where individual agents may interact either globally or even only at the local scale. 

The well known Cucker-Smale model has been proposed and studied in \cite{CS1, CS2} as a paradigmatic model for flocking, namely for modelling the evolution of dynamics where autonomous agents reach a consensus based on limited environmental information.
Consider $N\in \N$ agents and let $(x_i(t), v_i(t))\in\RR^{2d},$ $i=1,\dots, N,$ be their phase-space coordinates. One can think of $x_i(t)\in\RR^d$ by standing for the position of the $i^{\textrm{th}}$ agent and $v_i(t)\in\RR^d$ for its velocity, but for instance in social sciences these variables may stand for other notions such as opinions.
The general finite-dimensional Cucker-Smale model is the following:
\begin{equation}\label{standard}
\begin{split}
\dot x_i(t)&= v_i(t),\\
\dot v_i(t)&=\frac {\lambda}N\sum_{j=1}^N \psi_{ij}(t)(v_j(t)-v_i(t)),\qquad i=1,\dots, N,
\end{split}
\end{equation}
where the parameter $\lambda$ is a nonnegative coupling strength and the communication $\psi_{ij}(t)$ are of the form
\begin{equation*}\label{potential}
\psi_{ij}(t)=\psi (\vert x_i(t)-x_j(t)\vert ).
\end{equation*}
The function $\psi$ is called the potential.
Here and throughout, the notation $\vert\cdot\vert$ stands for the Euclidean norm in $\RR^d$.
Along any solution of (\ref{standard}), we define the (position and velocity) variances
\begin{equation*}\label{X}
X(t)=\displaystyle { \frac 1 {2 N^2}\sum_{i,j=1}^N\vert x_i(t)-x_j(t)\vert^2}
\end{equation*}
and
\begin{equation}\label{V}
\displaystyle {V(t)= {\frac{1}{2N^2}} \sum_{i,j=1}^N\vert v_i(t)-v_j(t)\vert^2}.
\end{equation}
\begin{Definition}{\rm
We say that a solution of (\ref{standard}) converges to consensus (or flocking) if
\begin{equation*}\label{cons}
\sup_{t>0} X(t)<+\infty\quad\quad\mbox {\rm and}\quad \quad \lim_{t\rightarrow +\infty} V(t)=0.
\end{equation*}
}\end{Definition}

The potential initially considered by Cucker and Smale in \cite{CS1, CS2} is the function $\psi(s)=\frac 1
{(1+s^2)^\beta }$ with $\beta\geq 0$. They proved that there  is
unconditional convergence to flocking whenever
 $\beta <\frac 1 2$.
If $\beta\geq 1/2$, there is convergence to flocking under appropriate assumptions on the values of the initial variances on positions and speeds (see \cite{HaHaKim}). Their analysis relies on a Lyapunov approach with quadratic functionals, which we will refer to in the sequel as an \textit{$L^2$ analysis}. This approach allows to treat symmetric communication rates.
An extension of the flocking result to the case of nonsymmetric communication rates has been proposed by Motsch and Tadmor \cite{MT}, with a different approach that we will refer to in the sequel as an \textit{$L^\infty$ analysis}, which we will describe further.

From the mathematical point of view, there have been a number of generalizations and of results on convergence to consensus for variants of Cucker-Smale models, involving more general potentials (friction, attraction-repulsion), cone-vision constraints, leadership (see \cite{Cavagna, CuckerDong, HaSlemrod, Mech, MT_SIREV, Vicsek, Yates}), clustering emergence (see \cite{Jabin, MT}), social networks (see \cite{Bellomo}), pedestrian crowds (see \cite{Cristiani, Lemercier}), stochastic or noisy models (see \cite{CuckerMordecki, HaLee}), kinetic models in infinite dimension (see \cite{Albi, Bellomo, canuto, carfor, degond, HT, Toscani}), and the control of such models (see \cite{Borzi, Caponigro1, Caponigro2, PRT, Wongkaew}).

\paragraph{Cucker-Smale with time-varying delays.}
In the present paper, we introduce time-delays in the Cucker-Smale model and we perform an asymptotic analysis of the resulting model. Time-delays reflect the fact that, for a given individual agent, information from other agents does not always propagate at infinite speed and is received after a certain time-delay, of reflect the fact that the agent needs time to elaborate its reaction to incoming stimuli.

We assume throughout that the delay $\tau (t) >0$ is time-varying. This models the fact that the amplitude of the delay may exhibit some seasonal effects or that it depends on the age of the agents for instance.
Our model is the following:
\begin{equation}\label{delayModel}
\boxed{
\begin{split}
\dot x_i(t) &= v_i(t),\\
\dot v_i(t) &= \frac {\lambda}N\sum_{\stackrel{j=1}{j\neq  i}}^N \psi_{ij}(t-\tau (t))(v_j(t-\tau (t) )-v_i(t)),\quad i=1,\dots, N,
\end{split}
}
\end{equation}
with initial conditions %, for $i=1,\dots, N,$
\begin{equation*}\label{IC}
x_i(t)=f_i(t), \quad % t\in [-\tau,0],\\
v_i(t)=g_i(t), \qquad t\in [-\tau (0),0],
\end{equation*}
where
 $f_i, g_i: [-\tau (0) , 0]\rightarrow \RR$ are given functions
and $\psi_{ij}(t),$ $i,j=1,\dots, N,$ are suitable communication rates.
In the symmetric case, we have
\begin{equation}\label{potential1}
\psi_{ij}(t)=\psi (\vert x_i(t)-x_j(t)\vert ),\quad i,j \in \{1,\dots, N\}.
\end{equation}
The time-delay function is assumed to be bounded: we assume that there exists $\overline \tau>0$ such that
\begin{equation}\label{tau1}
0\le\tau (t)\le\overline\tau ,\quad \forall\ t>0.
\end{equation}
We assume moreover that the function $t\mapsto \tau(t)$ is almost everywhere differentiable, and that there exists $c>0$ such that
\begin{equation}\label{tau3}
\vert\tau^\prime (t)\vert\le c<1, \quad \forall \ t>0.
\end{equation}
The potential function $\psi:[0, +\infty )\rightarrow (0, +\infty )$  in (\ref{potential1}) is assumed to be continuous and bounded.
Without loss of generality (if necessary, do a time reparametrization), we assume that
\begin{equation}\label{uno}
\psi_{ij}(t) \le 1, \quad \forall t\in [-\tau (0), +\infty ),\quad \forall i,j\in \{1,\dots, N\}.
\end{equation}
%This is guaranteed for instance if $\psi (s)<1, \ s\geq  0.$

Note that, in the model \eqref{delayModel} above, not only the delay is time-varying but also, more importantly, there is no delay in $v_i$ in the equation for velocity $v_i$. This assumption in our model is realistic because one expects that every agent receives information coming from the other agents with a certain delay while its own velocity is known exactly at every time $t$, but this makes the analysis considerably more complex, as we explain below.

\paragraph{State of the art.}
Simpler delay Cucker-Smale models have been considered in several contributions, with a constant delay $\tau >0.$

Firstly, a time-delayed model has been introduced and studied in \cite{Delay2}, where the equation for velocities (which actually also involves noise terms in that paper) is
 \begin{equation*}\label{geometric}
 \dot v_i(t)= \frac {\lambda}N\sum_{j=1}^N \psi_{ij}(t-\tau )(v_j(t-\tau )-v_i(t-\tau )) ,
\end{equation*}
with a constant delay $\tau$.
Considering $v_i(t-\tau)$ instead of $v_i(t)$ in the equation for $v_i(\cdot)$ is less natural because one can suppose that only the information on the velocities  of the other agents is known with a delay $\tau$.
However, this assumption in the model makes the analysis much easier because it allows to keep one of the most important features of the standard Cucker-Smale system \eqref{standard}, namely the fact that the mean velocity $\bar v(t)=\frac 1 N \sum_{i=1}^N v_i(t)$ remains constant, i.e., $\dot{\bar v}(t)=0$, as in the undelayed Cucker-Smale model. This fact significantly simplifies the arguments in the asymptotic analysis. 
In contrast, the mean velocity is not constant for our model \eqref{delayModel}, which makes the problem much more difficult to address.

Secondly, in \cite{Delay1} the authors consider as equation for the velocities
\begin{equation}\label{WU}
\dot v_i(t)= \alpha \sum_{j=1}^N a_{ij}(t-\tau) (v_j(t-\tau) -v_i (t)),
\end{equation}
where $\alpha>0$ and the coupling coefficients $a_{ij}$ are such that $\sum_{j=1}^N a_{ij}=1$, $i=1,\dots, N$. Compared with (\ref{delayModel}), the sum is running over \textit{all} indices, including $i$, and thus \eqref{WU} involves, with respect to \eqref{delayModel},  the additional term $a_{ii}(t-\tau )(v_i(t-\tau )-v_i(t))$ at the right-hand side. But on the one part, this term has no physical meaning. On the other part, the authors of \cite{Delay1} claim to study (\ref{delayModel}) but their claim is actually erroneous and their result (unconditional flocking for all delays) actually only applies to (\ref{WU}) (cf \cite[Eq. (7)]{Delay1}).
Note that (\ref{WU})  can be rewritten as
\begin{equation}\label{WU2}
\dot v_i(t)= \alpha \sum_{j=1}^N a_{ij}(t-\tau) v_j(t-\tau) -\alpha v_i (t),
\end{equation}
with a negative coefficient, independent of the time $t$, for the undelayed velocity $v_i(t)$ of the $i^{\textrm{th}}$ agent.
This allows to obtain a strong stability result: unconditional flocking for all time delays.

Thirdly, in the recent paper \cite{Choi}, the authors analyze a Cucker-Smale model with delay and normalized  communication weights $\Phi_{ij}$ given by
\begin{equation}\label{potentialChoi}
\Phi_{ij}(x,\tau )=\left\{
\begin{array}{l}
\displaystyle{
\frac {\psi(\vert x_j (t-\tau )-x_i(t)\vert )}
{\sum_{k\ne i} \psi (\vert x_k (t-\tau )-x_i(t)\vert )}
\quad \mbox {\rm if}\ j\ne i,}\\
0\hspace{4.8 cm}\mbox{\rm if}\ j=i,
\end{array}
\right.
\end{equation}
where the influence function $\psi$ is assumed to be bounded, nonincreasing, Lipschitz continuous on $[0, +\infty),$ with $\psi (0)=1.$
Thus, in practice, due to the assumption $\sum_{j=1}^N\Phi_{ij}=1$,
their model can be written as
\begin{equation*}\label{ChoiH}
\dot v_i(t)=  \sum_{j=1}^N \Phi_{ij}(x, t-\tau) v_j(t-\tau) -v_i (t),
\end{equation*}
to which the same considerations than for the model (\ref{WU2}) apply.
Moreover, the particular form of the communication weights $\Phi_{ij}$ allows to apply some convexity arguments in order to obtain the flocking result for sufficiently small delays. Then, the result strictly relies on the specific form of the interaction between the agents.
Note also that the influence function $\psi$ in the definition (\ref{potentialChoi}) of $\Phi_{ij}$ has as arguments $\vert x_k(t-\tau )-x_i(t)\vert,$ $k=1, \dots, N,$ $k\ne i,$ with the state of the $i^{\textrm{th}}$ agent at the time $t$ and the states of the other agents at time $t-\tau.$
This fact does not seem to have physical meaning, but it allows to easily derive the mean-field limit of the problem at hand by obtaining a nice and tractable kinetic equation.
In contrast, putting the time-delay also in the state of $i^{\textrm{th}}$ agent is more suitable to describe the physical model but it makes unclear (at least to us) the passage to mean-field limit (see Section \ref{Conclusion}).

\paragraph{Framework and structure of the present paper.}
In Section \ref{F}, we consider the model (\ref{delayModel}) with \textit{symmetric} interaction weights $\psi_{ij}$ given by (\ref{potential1}). In this symmetric case, we perform a $L^2$ analysis, designing appropriate quadratic Lyapunov functionals adapted to the time-delay framework. The main result, Theorem \ref{flock}, establishes convergence to consensus for small enough time-delays.

As in \cite{Delay2}, a structural assumption is required on the matrix of communication rates.
We define the $N\times N$ Laplacian matrix $L=(L_{ij})$ by
$$
L_{ij} = -\frac {\lambda} N\psi_{ij}, \quad \mbox{\rm for}\ i\neq  j,\quad\quad L_{ii}=\frac{\lambda}N\sum_{j\neq  i} \psi_{ij},
$$
with $\psi_{ij} = \psi(\vert x_i-x_j\vert))$.
The matrix $L$ is symmetric, diagonally dominant with nonnegative diagonal entries, has nonnegative eigenvalues, and its smallest eigenvalue is zero.
Considering the matrix $L(t)$ along a trajectory solution of \eqref{delayModel}, we denote by $\mu(t)$  its smallest positive eigenvalue, also called \textit{Fiedler number}.
The \textit{structural assumption} that we make throughout is the following:
\begin{equation}\label{structural}
\exists  \gamma >0 \ \mid \ \mu(t)\geq\gamma, \quad \forall t>0.
\end{equation}
This is guaranteed for instance if the communication rates are uniformly bounded away from zero, i.e., if there exists $\psi^*>0$ such that $\psi_{ij}(t)\geq  \psi^*$ for all $i,j$ and $t>0$ (but in that case of course there is unconditional convergence to consensus for the undelayed model).

In Section \ref{FL}, we consider the model (\ref{delayModel}) with possibly \textit{nonsymmetric} potentials:
\begin{equation*}
\begin{split}
\dot x_i(t) &= v_i(t),\\
\dot v_i(t) &= \frac{\lambda}{{N}} \sum_{j\neq  i} a_{ij}(t-\tau (t))(v_j(t-\tau (t) )-v_i(t)),\quad i=1,\dots, N,
\end{split}
\end{equation*}
where the communication rates $a_{ij}>0$ are arbitrary. They may of course be symmetric as above, e.g.
\begin{equation}\label{potsym}
a_{ij}(t)=  {\psi (\vert x_i(t)-x_j(t)\vert )},
 \end{equation}
or nonsymmetric, for instance
\begin{equation}\label{potnonsym}
a_{ij}(t)= \frac {N\psi (\vert x_i(t)-x_j(t)\vert )} {\sum_{k=1}^N \psi (\vert x_k(t)-x_i(t)\vert )},
\end{equation}
for a suitable bounded function $\psi.$
To analyze such models, we perform a $L^\infty$ analysis as in \cite{MT}, by considering, instead of Euclidean norms, the time-evolution of the diameters in position and velocity phase space. The main result, Theorem \ref{flock2}, establishes convergence to consensus under appropriate assumptions.

In Section \ref{Conclusion}, we provide a conclusion and further comments.

\section{Consensus for symmetric potentials: $L^2$ analysis}\label{F}
\setcounter{equation}{0}

\subsection{The main result}
\paragraph{Several notations.}
Following \cite{CS1}, we set
\begin{equation*}%\label{Delta}
\Delta = \left\{ (v_1,v_2, \ldots, v_N)\in(\RR^d)^N\ \mid\  v_1=\cdots=v_N\right\} = \left\{ (v,v, \ldots, v)\ \mid\ v\in \RR^d \right\} .
\end{equation*}
The set $\Delta$ is the eigenspace of $L$ associated with the zero eigenvalue.
Its orthogonal in $(\RR^d)^N$ is
$$
\Delta^\perp = \left\{ (v_1,v_2, \ldots, v_N)\in(\RR^d)^N\ \mid\  \sum_{i=1}^N v_i=0 \right\} .
$$
Given any ${\bf v}=(v_1,v_2, \ldots, v_N)\in(\RR^d)^N$, we denote the mean by $\bar v=\frac{1}{N}\sum_{j=1}^N v_j\in\RR^d$, and we define ${\bf w}= (w_1, \dots, w_N)\in (\RR^d)^N$ by
\begin{equation*}%\label{w}
w_i=v_i-\bar v, \quad i=1,\dots, N,
\end{equation*}
so that
$$
{\bf v}=(\bar v,\ldots,\bar v) + {\bf w} \in \Delta + \Delta^\perp,
$$
and we have  $L{\bf w}=L{\bf v}.$
Moreover,
\begin{equation}\label{1}
{\frac{1}{2N^2}}\sum_{i,j=1}^N \vert w_i-w_j\vert^2=\frac{1}{{N}}\Vert {\bf w}\Vert^2 ,
\end{equation}
and
\begin{equation}\label{2}
\langle L {\bf v}, {\bf v}\rangle =\frac 1 2 \frac {\lambda} N\sum_{i,j=1}^N\psi_{ij} \vert v_i-v_j\vert^2.
\end{equation}

\begin{Theorem}\label{flock}
Under the structural assumption $(\ref{structural})$, setting
\begin{equation}\label{threshold}
\tau_0= \frac {\gamma^2}{{2}\lambda^2 } \frac {1-c} {2\lambda^2+\gamma^2},
\end{equation}
if $\overline \tau^2 e^{\overline\tau }\in(0,\tau_0)$, then every solution of system $(\ref{delayModel})$ satisfies
\begin{equation}\label{Flock1}
V(t) \leq  C e^{-r t}\;,
\end{equation}
with
\begin{equation}\label{erre}
r=\min \Big \{\,\gamma- 4 \frac {\lambda^2}{\gamma} \frac {\lambda^2\overline \tau^2} {(1-c)e^{-\overline\tau }-2\lambda^2\overline \tau^2}\,, 1 \,\Big \},
\end{equation}
\begin{equation*}\label{valueC}
C= V(0) + \frac {\lambda^2\overline \tau}{\gamma N}\frac 1 {(1-c)e^{-\overline\tau }-2\lambda^2\overline \tau^2 } \int_{-\tau (0)}^0 e^s \int_s^0\sum_{i=1}^N \vert \dot v_i(\sigma )\vert^2 d\sigma\, ds.
\end{equation*}
\end{Theorem}

\begin{Remark}{\rm
The threshold $\tau_0$ (given by (\ref{threshold})) on the time-delay depends on the parameter $\lambda$ and on the lower bound $\gamma$ in (\ref{structural}) for the Fiedler number.
}\end{Remark}

\subsection{Proof of Theorem \ref{flock}}
We start with the following lemma.

\begin{Lemma}\label{cru}
We consider an arbitrary solution $({\bf x}(\cdot),{\bf v}(\cdot))$ of \eqref{delayModel}.
Setting
\begin{equation}\label{NUM}
R_\tau (t)=\frac 1 N \int_{t-\tau (t)}^t \sum_{i=1}^N \left \vert\dot v_i(s )
\right \vert^2 ds ,
\end{equation}
we have
\begin{equation}\label{stimacru}
\frac 1 N\sum_{i=1}^N \left \vert  \dot v_i(t)\right \vert^2
\leq  4\frac {\lambda^2} N \Vert {\bf w}(t)\Vert^2 +2\lambda^2\overline \tau R_\tau (t),
\end{equation}
for every $t>0$.
\end{Lemma}

\begin{proof}
Using (\ref{delayModel}), we compute
\begin{equation*}%\label{v1}
\begin{split}
\dot v_i (t) &=\frac {\lambda} N\sum_{j\neq  i} \psi_{ij}(t-\tau (t)) (w_j(t)-w_i(t))
+\frac {\lambda} N\sum_{j\neq  i} \psi_{ij}(t-\tau (t))(v_j(t-\tau (t)) -v_j(t))\\
&=
\frac {\lambda} N\sum_{j\neq  i} \psi_{ij}(t-\tau (t)) (w_j(t)-w_i(t))
-\frac {\lambda} N\sum_{j\neq  i} \psi_{ij}(t-\tau (t))\int_{t-\tau (t)}^t \dot v_j(s) ds.
\end{split}
\end{equation*}
Now, using (\ref{uno}), we get that
\begin{equation*}%\label{v2}
\left\vert\dot v_i (t)\right\vert \leq  \frac {\lambda} N\sum_{j\neq  i}\vert w_j(t)-w_i(t)\vert +\frac {\lambda} N \sum_{j\neq  i}
\int_{t-\tau (t) }^t \left\vert\dot v_j (s)\right\vert ds.
\end{equation*}
Then,
\begin{equation*}
\begin{split}
\left \vert\dot v_i (t)\right \vert^2
&\leq  2 \frac {\lambda^2}{N^2} \left(
\sum_{j=1}^N \vert w_i(t)-w_j(t)\vert
\right)^2+ 2 \frac {\lambda^2}{N^2} \left(
\sum_{j=1}^N\int_{t-\tau (t) }^t \left \vert
\dot v_j (s)
\right\vert ds
\right)^2  \\
&\leq
2\frac {\lambda^2}{N} \sum_{j=1}^N \vert w_j(t)-w_i(t)\vert^2+2\frac {\lambda^2}{N}
\sum_{j=1}^N \left (
\int_{t-\tau (t)}^t \left \vert
\dot v_j (s)
\right\vert ds
\right )^2 .
\end{split}
\end{equation*}
Using (\ref{1}), the Cauchy-Schwarz inequality and (\ref{tau1}), we infer that
\begin{equation*}%\label{v4}
\begin{split}
\sum_{i=1}^N\left \vert\dot v_i (t)\right \vert^2
&\leq  2\frac {\lambda^2} N\sum_{i,j=1}^N \vert w_i(t)-w_j(t)\vert^2 +2\lambda^2\tau (t)
\int_{t-\tau (t)}^t \sum_{i=1}^N \left \vert  \dot v_i(s)\right \vert^2 ds
\\
&\leq
4\lambda^2 \Vert {\bf w}(t)\Vert^2 +2\lambda^2\overline \tau \int_{t-\tau (t) }^t \sum_{i=1}^N
\left \vert  \dot v_i(s)\right \vert^2 ds,
\end{split}
\end{equation*}
which gives (\ref{stimacru}).
\end{proof}

\begin{Remark}{\rm
The term $R_\tau (t)$ is due to the presence of the time delay. Indeed, we
have two quantities  at the right-hand side of the inequality (\ref{stimacru}): the ``classical" term $\Vert w\Vert^2$ (coming from the undelayed model), and the term $R_\tau (t)$ caused by the delay effect.
}\end{Remark}

\begin{Lemma}
Given any solution $({\bf x}(\cdot),{\bf v}(\cdot))$ of \eqref{delayModel}, we have
\begin{equation}\label{N5statement}
\displaystyle{
\frac d {dt} \left ( \frac 1 N \Vert {\bf w}(t)\Vert^2\right )}
\displaystyle{\leq  - \frac { \gamma}{ N} \Vert {\bf w}(t)\Vert^2 +
\frac {\lambda^2\overline \tau } {\gamma}  R_\tau (t),
}
\end{equation}
for every $t>0$.
\end{Lemma}

\begin{proof}
Using \eqref{delayModel}, we compute
\begin{equation*}%\label{3}
\begin{split}
\dot w_i(t) &= \dot v_i(t)- \frac 1 N \sum_{k=1}^N\dot v_k(t) \\
&=\frac {\lambda} N\sum_{j\neq  i}\psi_{ij}(t-\tau (t) ) (v_j(t-\tau (t))-v_i(t))
-\frac {\lambda}{N^2} \sum_{k=1}^N \sum_{j\neq  k} \psi_{kj} (t-\tau (t) )(v_j(t-\tau (t))-v_k(t)) \\
&=
\frac {\lambda} N\sum_{j\neq  i}\psi_{ij}(t-\tau (t)) (v_j(t)-v_i(t))+
\frac {\lambda} N\sum_{j\neq  i}\psi_{ij}(t-\tau (t)) (v_j(t-\tau (t) )-v_j(t))
\\
&\quad
-\frac {\lambda}{N^2} \sum_{k=1}^N \sum_{j\neq  k} \psi_{kj} (t-\tau (t))(v_j(t)-v_k(t))
-\frac {\lambda}{N^2} \sum_{k=1}^N \sum_{j\neq  k} \psi_{kj} (t-\tau (t))(v_j(t-\tau (t))-v_j(t))
\\
&=
\frac {\lambda} N\sum_{j\neq  i}\psi_{ij}(t-\tau (t)) (w_j(t)-w_i(t))+
\frac {\lambda} N\sum_{j\neq  i}\psi_{ij}(t-\tau (t)) (v_j(t-\tau (t))-v_j(t))
\\
&\quad
-\frac {\lambda}{N^2} \sum_{k=1}^N \sum_{j\neq  k} \psi_{kj} (t-\tau (t))(w_j(t)-w_k(t))
-\frac {\lambda}{N^2} \sum_{k=1}^N \sum_{j\neq  k} \psi_{kj} (t-\tau (t))(v_j(t-\tau (t))-v_j(t)) .
\end{split}
\end{equation*}
Then,
\begin{equation*}%\label{4}
\sum_{i=1}^N w_i (t)\dot w_i (t)= -\frac 12\frac {\lambda} N\sum_{i,j=1}^N
\psi_{ij}(t-\tau (t) ) \vert w_i-w_j\vert^2
+\frac {\lambda} N
\sum_{i=1}^N \sum_{j\neq  i}   \psi_{ij}(t-\tau (t)) (v_j(t-\tau (t))-v_j(t)) w_i,
\end{equation*}
where we have used that $\sum_iw_i=0$ and
\begin{align*}
& \sum_{j\neq  i} \psi_{ij}(t-\tau (t))(w_j(t)-w_i(t))w_i(t) \\
=& \sum_{j\neq  i} \psi_{ij}(t-\tau (t))(w_j(t)-w_i(t))(w_i(t)-w_j(t)) + \sum_{j\neq  i} \psi_{ij}(t-\tau (t))(w_j(t)-w_i(t))w_j(t) \\
=& -\frac 12\sum_{j\neq  i} \psi_{ij}(t-\tau (t))\vert w_i(t)-w_j(t)\vert^2.
\end{align*}
Therefore, thanks to (\ref{2}), we infer that
\begin{equation*}%\label{5}
\begin{split}
&\frac d {dt} \left ( \frac 1 2 \Vert {\bf w}(t)\Vert^2\right )=-\langle L(t-\tau (t)) {\bf w}(t), {\bf w}(t)\rangle +\frac {\lambda} N
\sum_{i}\sum_{j\neq  i} \psi_{ij}(t-\tau (t)) (v_j(t-\tau (t))-v_j(t)) w_i(t) \\
&\quad =-\frac 12\frac {\lambda} N\sum_{i,j=1}^N
\psi_{ij}(t-\tau (t)) \vert w_i(t)-w_j(t)\vert^2+
\frac {\lambda} N
\sum_{i}\sum_{j\neq  i} \psi_{ij}(t-\tau (t)) (v_j(t-\tau (t))-v_j(t)) w_i(t).
\end{split}
\end{equation*}
The second term at the right-hand side of the above equality is bounded by
\begin{equation*}%\label{N1}
\Big \vert \frac {\lambda} N
\sum_{i}\sum_{j\neq  i} \psi_{ij}(t-\tau (t)) (v_j(t-\tau (t))-v_j(t)) w_i(t)\Big\vert\leq  \frac \lambda N \Vert w(t)\Vert \, \Vert U(t)\Vert,
\end{equation*}
where $U(t)=(U_1(t),\dots, U_N(t))$ is defined by
\begin{equation*}%\label{N2}
U_i(t)=\sum_{j\neq  i} \psi_{ij}(t-\tau (t)) (v_j(t-\tau (t))-v_j(t)), \quad i=1,\dots, N,
\end{equation*}
and is estimated by
\begin{equation*}%\label{N3}
\Vert U(t)\Vert\leq  \sum_{i=1}^N \vert U_i(t)\vert\leq  \sum_{i=1}^N \sum_{j\neq  i}\psi_{ij}(t-\tau (t))\int_{t-\tau (t)}^t \left\vert \dot v_j (s)   \right\vert ds
\leq  \sum_{ij}\psi_{ij}(t-\tau (t))\int_{t-\tau (t)}^t \left\vert \dot v_j (s)   \right\vert ds .
\end{equation*}
Therefore, we get
\begin{equation*}%\label{N4}
\begin{split}
\frac d {dt} \left ( \frac 1 2 \Vert {\bf w}(t)\Vert^2\right )
&\leq
-\frac 12\frac {\lambda} N\sum_{i,j=1}^N
\psi_{ij}(t-\tau (t)) \vert w_i(t)-w_j(t)\vert^2\\
&\qquad
+\frac {\lambda}N \sum_{ij}\psi_{ij}(t-\tau (t) )\int_{t-\tau (t)}^t \left\vert \dot v_j (s)   \right\vert ds\, \Vert {\bf w}(t)\Vert \\
&\leq
-\frac 12\frac {\lambda} N\sum_{i,j=1}^N
\psi_{ij}(t-\tau (t)) \vert w_i(t)-w_j(t)\vert^2 +\lambda\sum_{j=1}^N \int_{t-\tau (t)}^t \left\vert \dot v_j (s)\right\vert ds \, \Vert {\bf w}(t)\Vert \\
&\leq
-\frac 12\frac {\lambda} N\sum_{i,j=1}^N
\psi_{ij}(t-\tau (t) ) \vert w_i(t)-w_j(t)\vert^2 +\lambda \frac {\delta } 2 \Vert {\bf w}\Vert^2\\
&\qquad
+\frac{\lambda}{2\delta } \left(
\sum_{j=1}^N\int_{t-\tau (t)}^t \left\vert \dot v_j (s)\right\vert ds
\right)^2,
\end{split}
\end{equation*}
where we have used the Young inequality\footnote{This inequality states that, given any positive real numbers $a$, $b$ and $\delta$, we have $ab\leq \frac{a^2}{2\delta}+\frac{\delta b^2}{2}$.} for some arbitrary $\delta>0$.
Choosing $\delta=\frac {\gamma} {\lambda },$ where $\gamma$ is the constant in the structural assumption (\ref{structural}), we infer that
\begin{equation}\label{N5}
\begin{split}
\frac d {dt} \left ( \frac 1 2 \Vert {\bf w}(t)\Vert^2\right )
&\leq
-\langle L(t-\tau (t)) {\bf w}(t), {\bf w}(t)\rangle + \frac {\gamma} {2}  \Vert {\bf w}(t)\Vert^2 + \frac {\lambda^2}{2\gamma}\left(
\sum_{j=1}^N\int_{t-\tau (t)}^t \left\vert \dot v_j (s)\right\vert ds
\right)^2 \\
& \leq  - \frac {\gamma}{2 } \Vert {\bf w}(t)\Vert^2 +
\frac {\lambda^2\tau (t) } {2\gamma}\sum_{j=1}^N \int_{t-\tau (t)}^t \left \vert
\dot v_j(s) \right \vert^2 ds,
\end{split}
\end{equation}
which, using (\ref{tau1}) and the definition (\ref{NUM}) of $R_\tau(t),$ gives (\ref{N5statement}).
\end{proof}

We are now in a position to prove Theorem \ref{flock}.
Let $\beta>0$ be a positive constant to be chosen later.
We consider the Lyapunov functional along solutions of \eqref{delayModel}, defined by
\begin{equation}\label{Lyapunov}
{\mathcal L}(t) = %{\mathcal L} ({\bf v}(t))=
\frac 1 {2N} \Vert {\bf w}(t)\Vert^2+\frac {\beta } N\int_{t-\tau (t)}^t e^{-(t-s)}\int_s^t \sum_{i=1}^N
\left \vert  \dot v_i(\sigma )\right \vert^2 d\sigma\, ds,
\end{equation}
Using (\ref{N5}) and Lemma \ref{cru}, we have
\begin{equation*}
\begin{split}
\dot{\mathcal L}(t)
&\leq   - \frac {\gamma }{2 N} \Vert {\bf w}(t)\Vert^2 +\frac {\lambda^2\overline \tau }{2\gamma } R_\tau (t)
+\frac {\beta\tau (t)} N \sum_{i=1}^N \left \vert \dot v_i(t) \right \vert^2
-\frac {\beta} N (1-\tau^\prime (t))e^{-\tau(t)} \int_{t-\tau (t)}^t \sum_{i=1}^N\left \vert \dot v_j(s) \right \vert^2 ds \\
&\hspace{1 cm} -\frac {\beta }N \int_{t-\tau (t)}^t e^{-(t-s)}
\int_s^t \sum_{i=1}^N \vert \dot v_i(\sigma )\vert^2 d\sigma\, ds\\
& \leq  -\frac 1 N\left (\frac \gamma {2}-4\lambda^2 \beta\overline \tau   \right )
\Vert {\bf w}(t)\Vert^2-\left(  \beta (1-c) e^{-\overline\tau }- \frac {\lambda^2\overline \tau } {2\gamma}  -2\beta \lambda^2{\overline \tau}^2\right)
R_\tau(t)\\
& \hspace{1 cm} -\frac {\beta }N \int_{t-\tau (t)}^t e^{-(t-s)}
\int_s^t \sum_{i=1}^N \vert \dot v_i(\sigma )\vert^2 d\sigma\, ds,
\end{split}
\end{equation*}
where we have used (\ref{tau1})--(\ref{tau3}).
Convergence to consensus will then be ensured if
\begin{equation}\label{parametri}
\frac \gamma {2}-4\beta \lambda^2\overline \tau >0, \qquad
\beta (1-c) e^{-\overline\tau }-\frac {\lambda^2\overline \tau } {2\gamma} -2\beta \lambda^2{\overline \tau }^2\geq  0.
\end{equation}
The second inequality of (\ref{parametri}) yields a first restriction on the size of the delay, namely, that $ \overline\tau^2 e^{\overline \tau} < \frac {1-c}{2\lambda^2 }$.
Let us now choose the constant $\beta>0$ in the definition (\ref{Lyapunov}) of ${\mathcal L}(\cdot )$ so that both conditions in (\ref{parametri}) are satisfied:
\begin{equation*}\label{rangebeta}
\frac {\lambda^2\overline \tau }{2\gamma} \frac 1 {(1-c) e^{-\overline\tau }-2\lambda^2{\overline\tau}^2 }\le\beta<
\frac {\gamma }{8\lambda^2\overline\tau }.
\end{equation*}
This is possible only if
$$
\frac {\lambda^2\overline\tau^2}{(1-c) e^{-\overline\tau }-2\lambda^2\overline\tau^2 }< \frac {\gamma^2}{4\lambda^2},
$$
which is equivalent to
$$
\overline\tau^2 e^{\overline \tau } <\tau_0,
$$ with $\tau_0$ defined by \eqref{threshold}.
We conclude that, if $\overline \tau^2 e^{\overline\tau }<\tau_0,$ then we can choose $\beta$ such that
\begin{equation}\label{c1}
\frac {d\mathcal L}{dt}(t)\leq  - r {\mathcal L}(t),
\end{equation}
for a suitable positive constant $r$.
In particular, in order to obtain the best decay rate with our procedure, we fix
$\beta = \frac {\lambda^2 \overline \tau}{2\gamma} \frac 1 {(1-c)e^{-\overline\tau } -2\lambda^2\overline \tau^2 }$
obtaining
(\ref{c1}) with $r$ as in (\ref{erre}).

To conclude, it suffices to write that
\begin{equation*}%\label{c2}
\frac 1 N \Vert {\bf w}(t)\Vert^2\leq  2 {\mathcal L}(t)\leq  {2\mathcal L}(0) \, e^{-rt}.
\end{equation*}

Then (\ref{Flock1}) follows from the latter inequality, (\ref{V})  and  (\ref{1})
with $C={\mathcal L}(0)$ as in the statement.

\section{Consensus for nonsymmetric potentials: $L^\infty$ analysis}\label{FL}
\setcounter{equation}{0}

\subsection{The main result}

In this section, we consider nonsymmetric potentials, and we perform a $L^\infty$ analysis as in \cite{MT}.
We consider the Cucker-Smale system

\begin{equation}\label{delayModel2}
\begin{split}
\dot x_i(t) &= v_i(t),\\
\dot v_i(t) &= \frac{\lambda}{{N}} \sum_{j\neq  i} a_{ij}(t-\tau (t) )(v_j(t-\tau (t) )-v_i(t)),\quad i=1,\dots, N,
\end{split}
\end{equation}
with initial conditions, for $i=1,\ldots, N,$
\begin{equation*}%\label{IC2}
x_i(t)=f_i(t), \quad v_i(t)=g_i(t), \qquad t\in [-\tau (0),0],
\end{equation*}
where $f_i, g_i: [-\tau (0), 0]\rightarrow \RR$ are given functions and $a_{ij}>0$ quantifies the pairwise influence of $j^{\textrm{th}}$ agent on the alignment of  $i^{\textrm{th}}$ agent.
By rescaling $\lambda$ if necessary (or by time reparametrization), we assume that
\begin{equation}\label{normal}
\frac 1 N \sum_{j\neq  i} a_{ij}<1.
\end{equation}
This includes for instance the case considered in previous section, that is
$$
a_{ij}(t)=  {\psi (\vert x_i(t)-x_j(t)\vert )},
 $$
 with $\psi :[0, +\infty )\rightarrow [0, +\infty )$ satisfying $\psi (r)< 1$ for every $r\geq  0,$ but we can
consider a nonsymmetric interaction, for instance like in (\ref{potnonsym}),
$$
a_{ij}(t)= \frac {N\psi (\vert x_i(t)-x_j(t)\vert )} {\sum_{k=1}^N \psi (\vert x_k(t)-x_i(t)\vert )},
$$
for a suitable bounded function $\psi.$

As said before, an analogous delay model has been also investigated in \cite{Choi} for $\tau$ constant and under a  restrictive assumption on the potential interaction.
Indeed, the authors there consider the problem
\begin{equation}\label{Choi}
\begin{split}
\dot x_i(t) &= v_i(t),\\
\dot v_i(t) &=  \sum_{j=1}^N \Phi_{ij}(x, \tau )(v_j(t-\tau )-v_i(t)),\quad i=1,\dots, N,
\end{split}
\end{equation}
where the communication weights are defined in 
(\ref{potentialChoi}).
The choice of communication rates as in (\ref{potentialChoi}) simplifies significantly the model. Indeed, it follows that
$$\sum_{j=1}^N \Phi_{ij} =1,\quad \forall \ i=1,\dots, N,$$
and then one can rewrite the velocity equation as
$$\dot v_i(t) =  \sum_{j=i}^N \Psi_{ij}(x,\tau )v_j(t-\tau )-v_i(t),\quad i=1,\dots, N.$$
Namely, the term depending on $v_i$ in the left-hand side of (\ref{Choi}) is not
$\frac {\lambda} N\sum_{j\ne i}a_{ij}(t-\tau) v_i(t),$ as for the more general model (\ref{delayModel2}), but simply $-v_i(t).$
This simplifies the analysis allowing immediately to get a uniform bound for all times on the velocities of the agents if this bound is satisfied from the initial velocities
(see \cite[Lemma 2.1]{Choi}).

\medskip

Following \cite{MT}, we set $a_{ii}=N-\sum_{j\neq  i} a_{ij}$, so that $\sum_{j=1}^N a_{ij}=N$, $i=1, \ldots, N$.
Setting
\begin{equation}\label{vibar}
\tilde{v}_i(t)=\frac 1 N\sum_{j=1}^N a_{ij}(t-\tau (t)) v_j(t), \quad i=1, \ldots, N,
\end{equation}
the system (\ref{delayModel2}) is written as
\begin{equation*}\label{delayModelRewrite}
\begin{split}
\dot x_i(t)&= v_i(t),\\
\dot v_i(t)&=\lambda (\tilde{v}_i(t)-v_i(t))+\frac{\lambda}{N}\sum_{j\neq  i} a_{ij}(t-\tau (t))(v_j(t-\tau (t))-v_j(t)),\quad i=1,\dots, N.
\end{split}
\end{equation*}
We denote by $d_X(t)$ and $d_V(t)$ the diameter in position and velocity phase spaces (see \cite{HT}), respectively defined by
$$
d_X(t)=\max_{i,j}\vert x_j(t)-x_i(t)\vert, \qquad
d_V(t)=\max_{i,j}\vert v_j(t)-v_i(t)\vert.
$$
A solution of \eqref{delayModel2} converges to consensus if
\begin{equation*}%\label{defflocking}
\sup_{t\geq  0} d_X(t)<+\infty\quad\mbox{\rm and}\quad \lim_{t\rightarrow +\infty }d_V(t)=0.
\end{equation*}
Note that the functions $d_X$ and $d_V$ are not of class $C^1$ in general. We will thus use a suitable notion of generalized gradient, namely the upper Dini derivative, as in \cite{Delay1}, in order to perform our computations.
We recall that, for a given function $F$ continuous at $t$, the upper Dini derivative of $F$ at $t$ is defined by
$$
D^+F(t)=\limsup_{h\rightarrow 0^+}\frac {F(t+h)-F(t)} h.
$$
If $F$ is differentiable at $t,$ then $D^+F(t)=\frac {dF}{dt}(t).$ However, for all $t$ there exists a sequence $h_n\rightarrow 0^+$ such that
$$D^+F(t)=\lim_{n\rightarrow+\infty }\frac {F(t+h_n)-F(t)} {h_n}.$$
In particular, for a given t, there exist indices $r$ and $s$ such that
$d_X(t)=\vert x_r(t)-x_s(t)\vert $ and a sequence
$h_n\rightarrow 0^+$ for which
\begin{align*}
D^+d_X(t)&=\lim_{n\rightarrow\infty } h_n^{-1}
\left \{d_X(t+h_n)- \vert x_r(t)-x_s(t)\vert\right \}
\\
&\leq  \lim_{n\rightarrow\infty } h_n^{-1}
\left \{\vert x_r(t+h_n)-x_s(t+h_n)\vert- \vert x_r(t)-x_s(t)\vert\right \}
\leq\Big\vert \frac {dx_r}{dt}(t)-\frac {dx_s}{dt}(t)\Big\vert.
\end{align*}
Analogous arguments apply to $D^+d_V(t)$ and $D^+d^2_V(t).$

\begin{Theorem}\label{flock2}
We assume that there exists $\psi^*>0$ such that
\begin{equation}\label{structural2}
\frac 1 {N^2} \sum_{i,j=1}^N\min \left( a_{qi}a_{pj}, a_{qj}a_{pi}\right) \geq  {\psi^*}, \quad\forall\  p, q =1,\dots, N.
\end{equation}
Setting
\begin{equation}\label{def2tau0}
\tau_0= \frac {1-c} {\lambda}\,\frac {{\psi^*}}{{\psi^*}+2},
\end{equation}
if $\overline \tau e^{\overline\tau }\in(0,\tau_0)$, then every solution of $(\ref{delayModel2})$ satisfies
\begin{equation}\label{exp2}
d_V(t)\leq  C e^{-rt}, \quad t\geq  0,
\end{equation}
with
\begin{equation}\label{rsecondo}
r=\min \Big \{ \, \lambda \Big ( {\psi^*} -\frac {2\lambda\overline\tau }{(1-c)e^{-\overline\tau }-\lambda\overline\tau }  \Big )\,, 1\, \Big\},
\end{equation}
\begin{equation}\label{Csecondo}
C=d_V(0)+\frac {2\lambda}{(1-c) e^{-\overline\tau }-\lambda\overline\tau }\int_{-\tau (0)}^0 e^s \int_s^0\max_{j=1,\dots, N}\left\vert \dot v_j(\sigma)\right\vert d\sigma\, ds.
\end{equation}
\end{Theorem}

Assumption $(\ref{structural2})$ is done in order to ensure an unconditional convergence result. It is satisfied for instance, for the interactions (\ref{potsym}) and (\ref{potnonsym}) if the influence function $\psi$ in the definitions of $a_{ij}$ satisfies the lower bound $\psi (r)\geq  \psi_0>0.$ Indeed, $(\ref{structural2})$ is verified in both cases with $\psi^*={\psi_0}^2$ and
$\psi^*= \left (\frac {\psi_0} {\Vert\psi\Vert_\infty}\right )^2$ respectively.

\subsection{Proof of Theorem \ref{flock2}}
We start by establishing several estimates.

\begin{Lemma}\label{MT}
$(\cite{MT}).$ Let $S=(S)_{1\leq i,j\leq N}$ be a skew-symmetric matrix such that $\vert S_{ij}\vert\leq  M$ for all $i,j$. Let $u, w\in \RR^N$ be two given real vectors with nonnegative entries, $u_i, w_i\geq  0,$ and let $\overline U =\frac 1 N\sum_i u_i$ and $\overline W=\frac 1 N\sum_i w_i$.
Then,
\begin{equation}\label{MTsimplier}
\frac 1 {N^2}\vert \langle Su,w\rangle \vert \leq  M\Big (\overline U\overline W -\frac 1 {N^2}\sum_{i,j=1}^N\min\,
\left( u_iw_j, u_jw_i\right) \Big ).
\end{equation}
\end{Lemma}

\begin{Lemma}\label{derivate}
Let $(({\bf x}(\cdot),{\bf v}(\cdot))$ be a solution of $(\ref{delayModel2}).$
Setting
\begin{equation}\label{resto2}
\sigma_\tau (t)=\int_{t-\tau (t) }^t \max_{j=1,\dots, N}\left \vert \dot v_j   (s)\right \vert\, ds ,
\end{equation}
we have, for every $t\geq  0,$
\begin{equation}\label{stimadV}
D^+d_X(t)\leq  d_V(t),\qquad
D^+d_V(t)\leq  -\lambda {\psi^*} d_V(t)+2\lambda \sigma_\tau (t) ,
\end{equation}
where $\psi^*$ is the constant in $(\ref{structural2})$.
\end{Lemma}

\begin{proof}
Fix $t\geq  0$ and let $p, q, r $ and $s$ be indices such that $d_X(t)=\vert x_r(t)-x_s(t)\vert$ and $d_V(t)=\vert v_p(t)-v_q(t)\vert$. Then we have $D^+d_X(t)\leq  \vert v_r(t)-v_s(t)\vert\leq  d_V(t)$, and besides,
\begin{equation*}%\label{V1}
\begin{split}
D^+(d_V^2 (t)) &\leq  2\left\langle v_p-v_q, \dot v_p(t)-\dot v_q(t)\right\rangle \\
&= 2\lambda \langle v_p(t)-v_q(t), \tilde v_p(t)-\tilde v_q(t)\rangle -2\lambda\vert v_p(t)-v_q(t)\vert^2\\
&Ê\quad +2\frac {\lambda }N\Big \langle v_p(t)-v_q(t), \sum_{j\neq  p} a_{pj}(t-\tau (t) )(v_j(t-\tau (t))-v_j(t)) \\
&\qquad\qquad\qquad\qquad\qquad\qquad -\sum_{j\neq  q} a_{qj}(t-\tau (t)) (v_j(t-\tau (t) )-v_j(t))\Big\rangle ,
\end{split}
\end{equation*}
where $\tilde v_i,$ $i=1,\dots, N,$ are defined in (\ref{vibar}),
and then
\begin{equation*}%\label{V2}
\begin{split}
D^+(d_V^2 (t))
& \leq  2\left\langle v_p(t)-v_q(t), \dot v_p(t)-\dot v_q(t) \right\rangle  \\
& = 2\lambda \langle v_p(t)-v_q(t), \tilde v_p(t)-\tilde v_q(t)\rangle -2\lambda\vert v_p(t)-v_q(t)\vert^2 \\
&\qquad\qquad\qquad\qquad\qquad\qquad +4\lambda \vert v_p(t)-v_q(t)\vert\int_{t-\tau (t) }^t\max_{j=1, \dots, N}\left\vert \dot v_j  (s)\right\vert\, ds .
\end{split}
\end{equation*}
But since
\begin{equation*}%\label{p1}
\begin{split}
\tilde v_p(t)-\tilde v_q(t) &= \frac 1 N\sum_{j=1}^N a_{pj}(t-\tau (t) ) v_j(t)- \frac 1 N\sum_{i=1}^N a_{qi}(t-\tau (t)) v_i(t)\\
& = \frac 1 {N^2}\sum_{i=1}^N a_{qi}(t-\tau (t) )\sum_{j=1}^N a_{pj}(t-\tau (t) ) v_j(t)\\
&\qquad - \frac 1 {N^2}\sum_{j=1}^N a_{pj}(t-\tau (t) ) \sum_{i=1}^N a_{qi}(t-\tau (t) ) v_i(t) \\
& =  \frac 1 {N^2}\sum_{i, j=1}^N a_{qi}(t-\tau (t) ) a_{pj}(t-\tau (t) ) (v_j(t)- v_i(t)) ,
\end{split}
\end{equation*}
%Then, substituting (\ref{p1}) in (\ref{V2}), we have
we get that
\begin{equation}\label{p2}
\begin{split}
D^+(d_V^2(t))\leq &
\ 2\frac {\lambda }{N^2} \sum_{i, j=1}^N a_{qi}(t-\tau (t) ) a_{pj}(t-\tau (t) ) \langle v_j(t)- v_i(t), v_p(t)-v_q(t)\rangle \\
& -2\lambda\vert v_p(t)-v_q(t)\vert^2 +4\lambda \vert v_p(t)-v_q(t)\vert\int_{t-\tau (t) }^t\max_{j=1,\dots, N}\left\vert \dot v_j  (s)\right\vert\, ds.
\end{split}
\end{equation}
We estimate the first term at the right-hand side of (\ref{p2}) by applying Lemma \ref{MT}, with
$S_{ij}=\langle v_j(t)-v_i(t), v_p(t)-v_q(t)\rangle$ and $ u_i=a_{qi}(t-\tau )$ and $w_j= a_{pj}(t-\tau (t))$ for $i,j=1, \ldots, N$.
Since $\vert S_{ij}\vert \leq  d_V^2(t)$ and $\overline U, \overline W=1,$ using Assumption (\ref{structural2}), we infer from (\ref{MTsimplier}) with $\theta = {\psi^*}$ that
\begin{equation*}%\label{p3}
\Big\vert \frac 1 {N^2}\sum_{i, j=1}^N a_{qi}(t-\tau (t)) a_{pj}(t-\tau (t)) \langle v_j(t)-v_i(t), v_p(t)-v_q(t)\rangle\Big\vert \leq  (1-{\psi^*}) d_V^2(t).
\end{equation*}
With the above estimate, we obtain from (\ref{p2}) that
\begin{equation*}%\label{p4}
D^+(d_V^2(t))\leq  2\lambda (1-{\psi^*}) d_V^2(t)-2\lambda d_V^2(t) +4\lambda d_V(t)\sigma_\tau (t),
\end{equation*}
from which (\ref{stimadV}) follows.
\end{proof}

\begin{Lemma}\label{stimapunt}
Let $({\bf x}(\cdot),{\bf v}(\cdot))$ be a solution of $(\ref{delayModel2})$. Then
\begin{equation}\label{SP}
\max_{j=1, \dots, N} \left\vert \dot v_j  (t)\right\vert \leq  \lambda d_V(t)+\lambda  \sigma_\tau (t),
\end{equation}
for every $t\geq 0$.
\end{Lemma}

\begin{proof}
Using (\ref{delayModel2}), we have
\begin{equation*}%\label{p5}
\dot v_i  (t)=\frac {\lambda }N \sum_{j\neq  i}a_{ij}(t-\tau (t))(v_j(t)-v_i(t))  + \frac {\lambda }N \sum_{j\neq  i}a_{ij}(t-\tau (t) )(v_j(t-\tau (t))-v_j(t)),
\end{equation*}
from which we infer that
\begin{equation*}%\label{p6}
\left\vert \dot v_i  (t)\right\vert \leq
\frac {\lambda }N \sum_{j\neq  i}a_{ij}(t-\tau (t)) d_V(t)+ \frac
{\lambda }N \sum_{j\neq  i}a_{ij}(t-\tau (t))\int_{t-\tau (t)}^t \left\vert \dot v_j  (s)\right\vert  \,ds .
\end{equation*}
Then, we have
$$\left\vert \dot v_i (t)\right\vert \leq
\lambda d_V(t)+\lambda \int_{t-\tau (t)}^t \max_{j=1,\dots, N}\left\vert \dot v_j  (s)\right\vert  \,ds,$$
and the lemma is proved by taking the maximum in the left-hand side and using the definition (\ref{resto2}) of $\sigma_\tau (t).$
\end{proof}

We are now in a position to prove the theorem. Let $\beta >0$ to be chosen later. We consider the Lyapunov functional defined along any solution by
\begin{equation}\label{LyaF}
\mathcal{F} (t)=  d_V(t)+\beta \int_{t-\tau (t)}^t e^{-(t-s)}\,\int_s^t \max_{j=1,\dots,N}\left\vert \dot v_j  (\sigma )\right\vert  \, d\sigma\,ds.
\end{equation}

First of all, using (\ref{stimadV}), we have
$$\begin{array}{l}
\displaystyle{
D^+\mathcal{F}(t)\leq  -\lambda {\psi^*} d_V(t)+2\lambda \sigma_\tau (t)}\\
\hspace{1,4 cm} \displaystyle{ -\beta (1-\tau^\prime (t)) e^{-\tau (t)} \int_{t-\tau (t) }^t \max_{j=1,\dots, N} \left\vert \dot v_j  (s)\right\vert \, ds +\beta \tau (t)
\max_{j=1,\dots, N} \left\vert \dot v_j  (t)\right\vert}\\
\hspace{2 cm} \displaystyle{-\beta \int_{t-\tau (t)}^t e^{-(t-s)}\,\int_s^t \max_{j=1,\dots,N}\left\vert \dot v_j  (\sigma )\right\vert  \, d\sigma\,ds.}
\end{array}
$$
It follows from Lemma \ref{stimapunt}, using (\ref{tau1})-(\ref{tau3}), that
\begin{equation*}%\label{p8}
\begin{split}
D^+\mathcal{F}(t) &\leq  -\lambda {\psi^*}d_V(t)+ (2\lambda -\beta (1-c)e^{-\overline\tau })\sigma_\tau (t)+\beta\overline \tau\lambda  d_V(t)\\
&\hspace{1 cm}
+\beta\overline \tau\lambda \int_{t-\tau (t)}^t \max_{j=1, \dots, N}\left\vert \dot v_j  (s)\right\vert\, ds 
-\beta \int_{t-\tau (t)}^t e^{-(t-s)}\,\int_s^t \max_{j=1,\dots,N}\left\vert \dot v_j  (\sigma )\right\vert  \, d\sigma\,ds
\\
& \leq  -\lambda ({\psi^*} -\beta\overline \tau )d_V(t)- (\beta (1-c)e^{-\overline\tau } -2\lambda - \beta\overline \tau\lambda  ) \sigma_\tau (t)\\
&\hspace{1 cm}-\beta \int_{t-\tau (t)}^t e^{-(t-s)}\,\int_s^t \max_{j=1,\dots,N}\left\vert \dot v_j  (\sigma )\right\vert  \, d\sigma\,ds.
\end{split}
\end{equation*}
Convergence to consensus will be ensured if
\begin{equation}\label{condit.parametri}
 {{\psi^*}} -\overline \tau\beta >0, \qquad
\beta (1-c)e^{-\overline\tau }-2\lambda -\beta\overline \tau\lambda  \geq  0.
\end{equation}
The second inequality of (\ref{condit.parametri}) yields a first restriction on the size of the delay: $\overline \tau  e^{\overline\tau }<\frac {1-c} {\lambda }$. Let us now choose the constant $\beta>0$ in the definition (\ref{LyaF}) of the Lyapunov functional $\mathcal{F}$ so that both conditions in \eqref{condit.parametri} are satisfied. We impose that $\beta < \frac {{\psi^*}}{\overline\tau }$ and $\beta \geq  \frac {2\lambda }{ (1-c)e^{-\overline\tau }-\lambda\overline \tau }$. This is possible if and only if $\frac {2\lambda }{ (1-c)e^{-\overline\tau }-\lambda\overline \tau } < \frac {{\psi^*}}{\overline \tau }$, that is, equivalently, $\overline\tau e^{\overline\tau } <\tau_0$, where $\tau_0$ is defined by \eqref{def2tau0}.

We now choose $\beta$ in the definition of $\mathcal{F}$ such that
\begin{equation}\label{richiama}
D^+\mathcal{F}(t)\leq  -r {\mathcal F}(t),
 \end{equation}
for a suitable positive constant $r$. In order to have a better decay estimate let us fix $\beta= \frac {2\lambda }
{(1-c)e^{-\overline\tau }-\lambda\overline\tau }.$ Then, we obtain (\ref{richiama}) with $r$ as in
(\ref{rsecondo}).
 Therefore,
$$d_V(t)\leq  \mathcal{F}(t)\leq  \mathcal{F}(0) e^{-r t} \,,\quad t\ge 0\,.$$
The exponential decay estimate (\ref{exp2}) is then proved with $C=\mathcal {F}(0)$ as in (\ref{Csecondo}).

\section{A numerical simulation}
We provide here a numerical simulation illustrating our results. We take $d=2$, $N=3$ ($3$ agents), and we take the Cucker and Smale potential $\psi(s)=\frac{1}{(1+s^2)^\beta}$ with $\beta=2$. As recalled in the introduction, for such a value of $\beta$ convergence to consensus does not occur for any initial condition.

For the moment, we do not consider any time delay in the model, i.e., $\tau\equiv 0$. We take as initial conditions
\begin{align*}
&x_1^0 = (0,0),\quad v_1^0 = (1,0),\\
&x_2^0 = (0,1),\quad v_2^0 = (1,0),\\
&x_2^0 = (1,0),\quad v_2^0 = (0.5,0.5).
\end{align*}
For such initial conditions, we have convergence to consensus, see Figures \ref{fig_0_10} and \ref{fig_0_50}. On these figures, the initial points are represented with a red point. At the top left are drawn the curves $t\mapsto x_i(t)\in\mathbb{R}^2$: motion in the plane of the three agents. At the top right, one can see the modulus of the speeds $\Vert v_i(t)\Vert$, in function of $t$. AT the bottom are drawn the time evolution of the position variance $X(t)$ and of the speed variance $V(t)$.

\begin{figure}[H]
\centerline{\includegraphics[width=12cm]{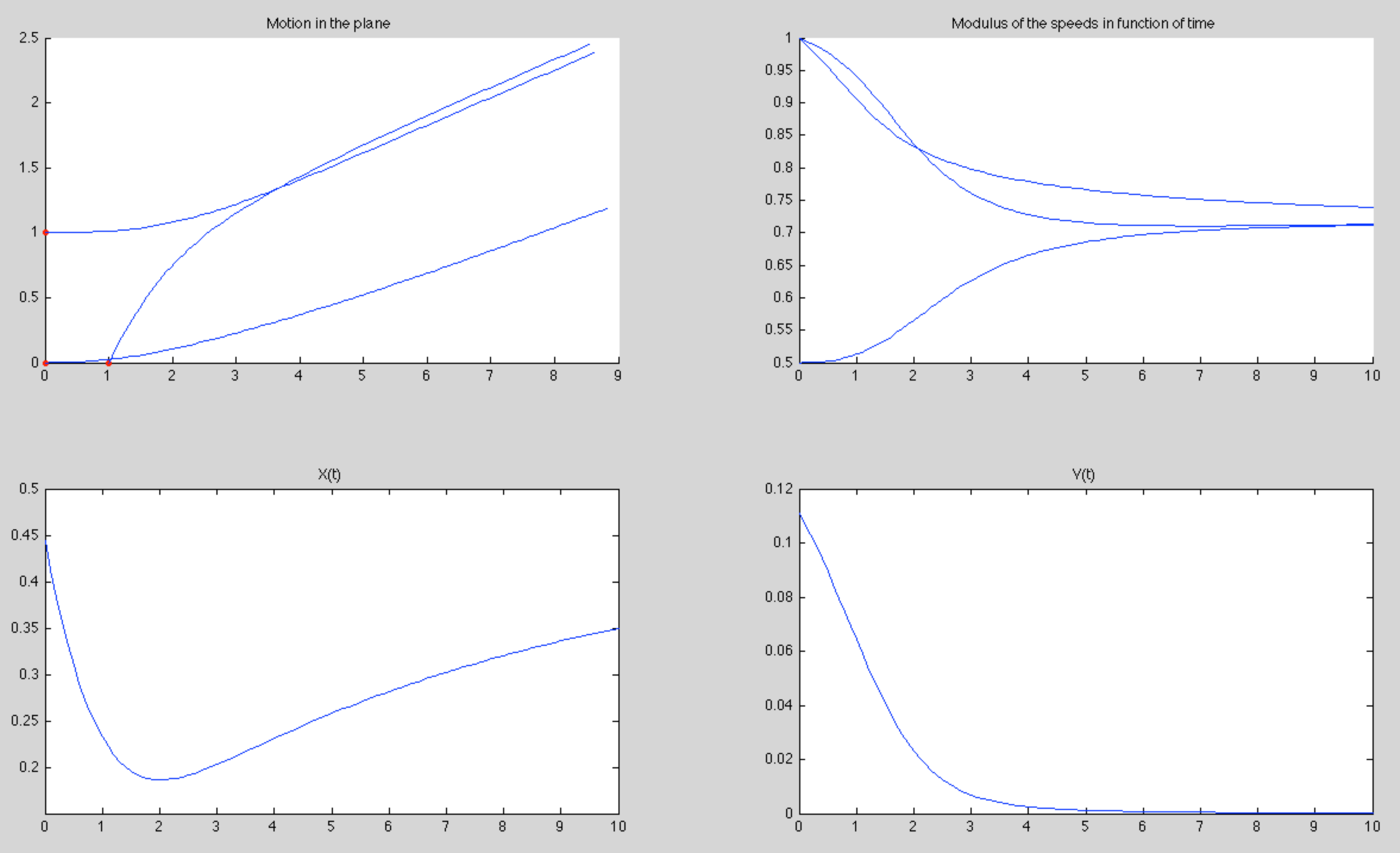}}
\caption{No time delay, $\tau=0$. Simulation on the time interval $[0,10]$.}
\label{fig_0_10}
\end{figure}
\begin{figure}[H]
\centerline{\includegraphics[width=12cm]{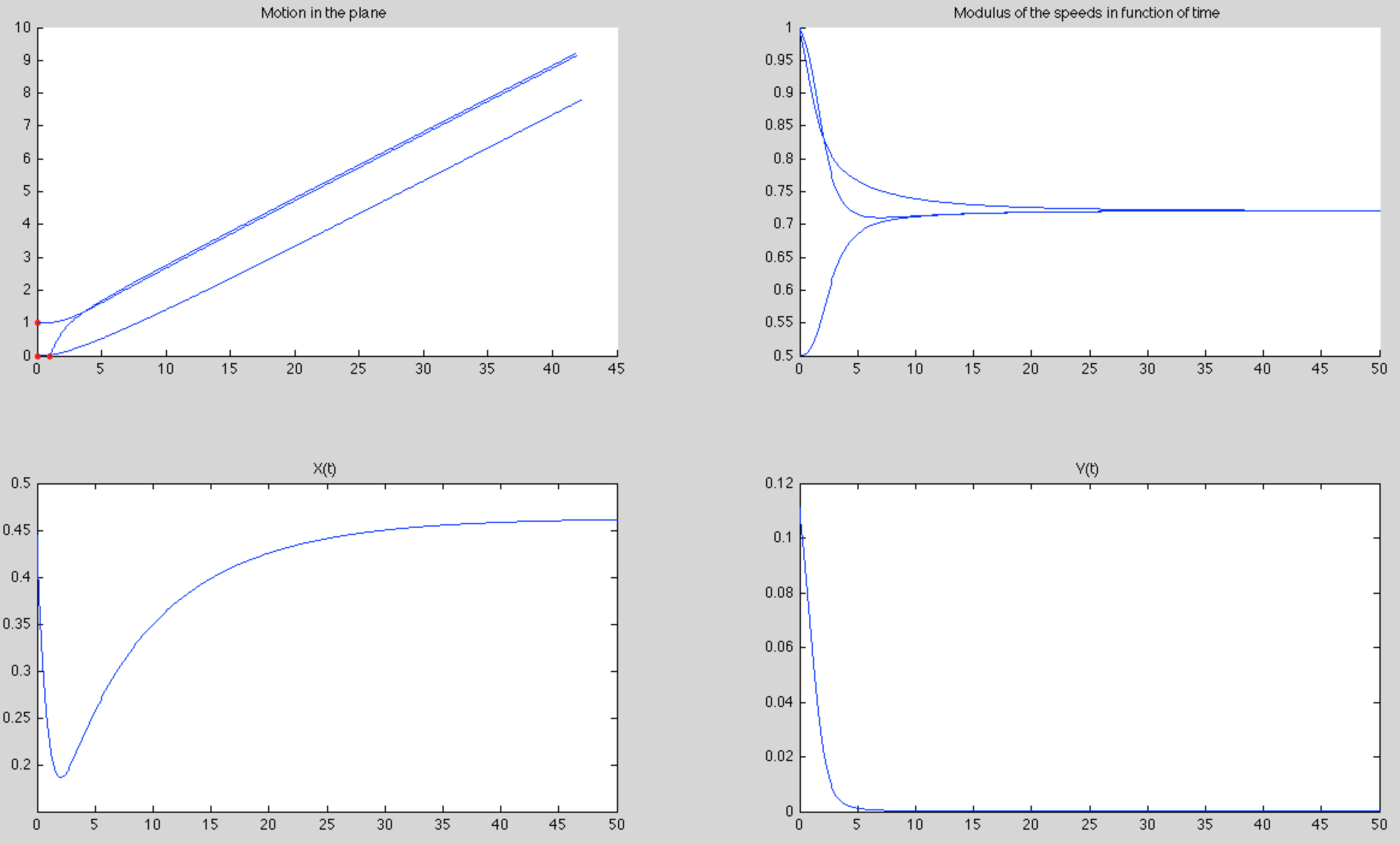}}
\caption{No time delay, $\tau=0$. Simulation on the time interval $[0,50]$.}
\label{fig_0_50}
\end{figure}

We now introduce a time delay which, for simplicity, we take fixed: $\tau(t)\equiv\tau$. We take as initial conditions, on $[-\tau,0]$,
$$
x_i(t) = x_i^0 + (t+\tau) v_i^0, \quad v_i(t)=v_i^0,\qquad i=1,\ldots,N.
$$
In other words, along the interval $[-\tau,0]$ the agents follow the dynamics $\dot x_i=v_i$ and $\dot v_i=0$, and thus each agent performs a translation motion, starting at $x_i^0$ with the speed $v_i^0$.

The corresponding solution for $\tau=5$ is drawn on Figure \ref{fig_5_20}. For this value of the time delay, convergence to consensus is lost. When time goes to infinity, the agents do not remain grouped, and one can indeed observe that the position variance $X(t)$ tends to $+\infty$.

\begin{figure}[H]
\centerline{\includegraphics[width=12cm]{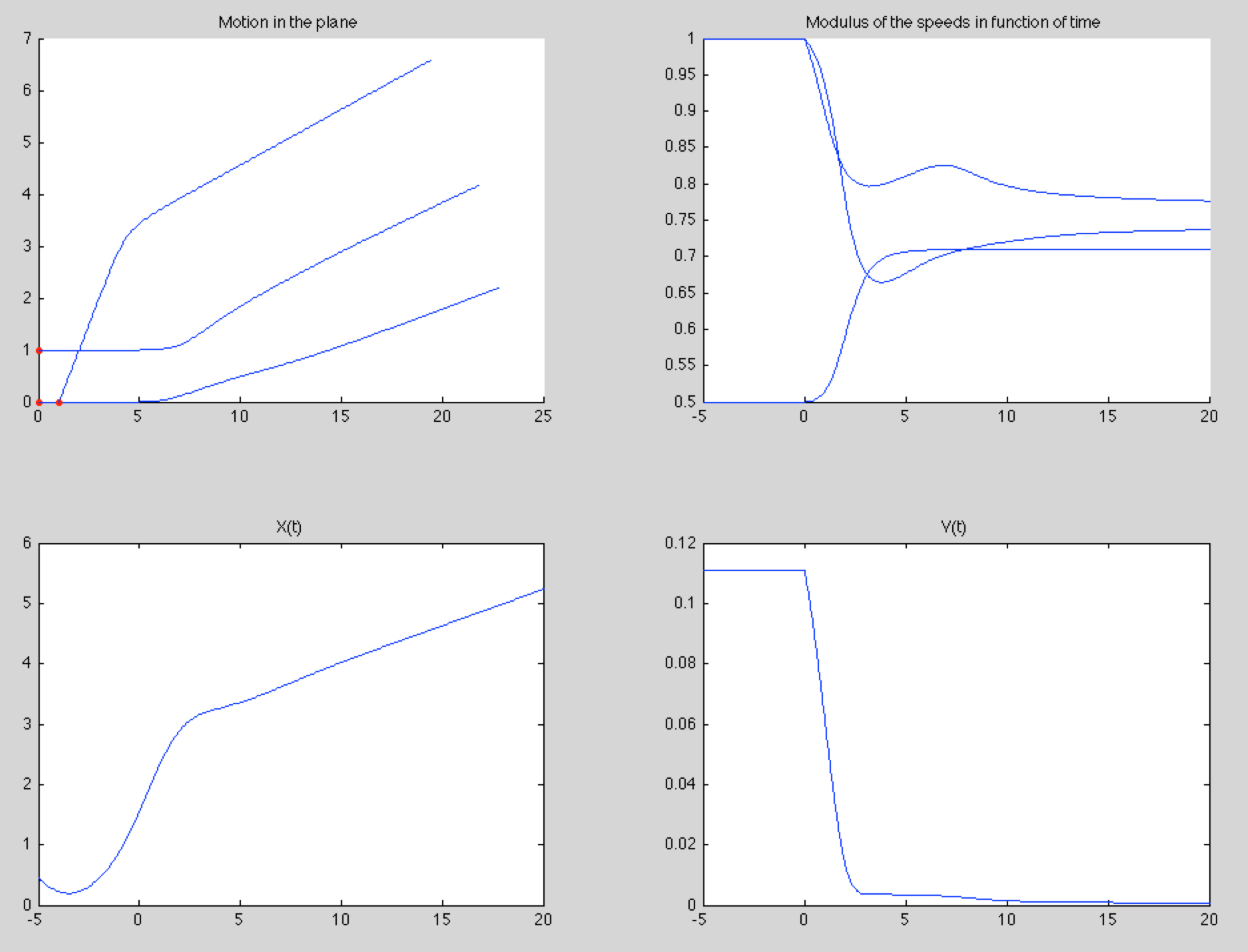}}
\caption{Time delay $\tau=5$. Simulation on the time interval $[-5,20]$.}
\label{fig_5_20}
\end{figure}

The loss of consensus actually occurs for smaller values of $\tau$, but we chose here to provide a simulation for $\tau=5$ because, if we take $\tau$ smaller, we have to consider much larger integration times to see that consensus is lost, and simulations are then not so nice to be printed here.

Numerically, we find that, for $\tau$ less than (approximately) $0.5$, consensus still occurs, whereas for larger values of $\tau$ consensus is lost. This threshold is slightly larger than the threshold $\tau_0$ predicted by our result (which is not sharp).

\section{Conclusion and further comments}\label{Conclusion}

We have analyzed the finite-dimensional general Cucker-Smale model with time-varying time-delays, and we have established precise convergence results to consensus under appropriate assumptions on the time-delay function $\tau(\cdot)$. Our results are valid for symmetric as well as for nonsymmetric interaction rates. The symmetric case has been analyzed thanks to a $L^2$ analysis, in the spirit of the original papers \cite{CS1, CS2}, while we were able to deal with the loss of symmetry by carrying out a $L^\infty$ analysis as in \cite{MT}. 

In both cases, we have established convergence to consensus provided the time-delay is below a precise threshold. The bound depends on the  coupling strength $\lambda$, on the communication weights and on the bound $c$ on the time-derivative of $\tau(\cdot)$, but it does not depend on the number $N$ of the agents.
This important fact suggests that it might be possible to extend our analysis performed here on the finite-dimensional Cucker-Smale model to the infinite-dimensional case. 

\paragraph{Towards a kinetic extension.}
The kinetic equation for the undelayed Cucker-Smale model has been derived in \cite{HT} using the BBGKY hierarchy from the Cucker-Smale particle model as a mesoscopic description for flocking (see also \cite{HL, PRT}).
By considering the mean-field limit in the case $\tau=0,$ one obtains the kinetic equation
$$
\partial_t\mu + \langle v,\mathrm{grad}_x\mu\rangle + \mathrm{div}_v \left( (\xi[\mu])\mu \right) = 0 ,
$$
where $\mu(t) = \mu(t,x,v)$ is the density of agents at time $t$ at $(x,v),$ 
with the interaction field defined by
$$
\xi[\mu](x,v) = \int_{\RR^d\times\RR^d} \psi(\vert x-y\vert) (w-v) \, d\mu(y,w) .
$$
If we introduce a delay $\tau$ in the Cucker-Smale system, even when $\tau$ is constant, it is no clear how to deduce the corresponding kinetic model.
In contrast, it is easy to pass to the mean-field limit when one considers a Cucker-Smale model with communication weights as in (\ref{potentialChoi}): indeed, the authors of \cite{Choi},  putting a delay on $x_j$ but not on $x_i$ in the communication weights in the equation for $v_i,$ are able to pass to the mean-field limit and obtain the kinetic equation
$$
\partial_t\mu(t) + \langle v,\mathrm{grad}_x\mu(t)\rangle + \mathrm{div}_v \left( (\xi[\mu(t-\tau)])\mu(t) \right) = 0 .
$$
Deriving an appropriate kinetic equation by considering the mean-field limit of (\ref{delayModel}), with communication weights depending on the states at time $t-\tau $ for all the agents, as it is, in our opinion, more adequate from a physical point of view, seems out of reach at this moment. We let it as an open question.

\end{document}